\documentclass[12pt]{amsart}
\usepackage{verbatim}
\usepackage{amsthm}
\usepackage{amsmath}
\usepackage{amsfonts}
\usepackage{amssymb}
\begin{document}
\numberwithin{equation}{section}

\def\1#1{\overline{#1}}
\def\2#1{\widetilde{#1}}
\def\3#1{\widehat{#1}}
\def\4#1{\mathbb{#1}}
\def\5#1{\frak{#1}}
\def\6#1{{\mathcal{#1}}}

\newcommand{\de}{\partial}
\newcommand{\R}{\mathbb R}
\newcommand{\al}{\alpha}
\newcommand{\tr}{\widetilde{\rho}}
\newcommand{\tz}{\widetilde{\zeta}}
\newcommand{\tv}{\widetilde{\varphi}}
\newcommand{\tO}{\widetilde{\Omega}}
\newcommand{\hv}{\hat{\varphi}}
\newcommand{\tu}{\tilde{u}}
\newcommand{\usc}{{\sf usc}}
\newcommand{\tF}{\tilde{F}}
\newcommand{\debar}{\overline{\de}}
\newcommand{\Z}{\mathbb Z}
\newcommand{\C}{\mathbb C}
\newcommand{\Po}{\mathbb P}
\newcommand{\zbar}{\overline{z}}
\newcommand{\G}{\mathcal{G}}
\newcommand{\So}{\mathcal{U}}
\newcommand{\Ko}{\mathcal{K}}
\newcommand{\U}{\mathcal{U}}
\newcommand{\B}{\mathbb B}
\newcommand{\oB}{\overline{\mathbb B}}
\newcommand{\Cur}{\mathcal D}
\newcommand{\Dis}{\mathcal Dis}
\newcommand{\Levi}{\mathcal L}
\newcommand{\SP}{\mathcal SP}
\newcommand{\Sp}{\mathcal Q}
\newcommand{\Ma}{\mathcal M}
\newcommand{\Co}{\mathcal C}
\newcommand{\Hol}{{\sf Hol}(\mathbb H, \mathbb C)}
\newcommand{\Aut}{{\sf Aut}(\mathbb D)}
\newcommand{\D}{\mathbb D}
\newcommand{\oD}{\overline{\mathbb D}}
\newcommand{\oX}{\overline{X}}
\newcommand{\loc}{L^1_{\rm{loc}}}
\newcommand{\loci}{L^\infty_{\rm{loc}}}
\newcommand{\la}{\langle}
\newcommand{\ra}{\rangle}
\newcommand{\thh}{\tilde{h}}
\newcommand{\N}{\mathbb N}
\newcommand{\kd}{\kappa_D}
\newcommand{\Hr}{\mathbb H}
\newcommand{\ps}{{\sf Psh}}
\newcommand{\tg}{\widetilde{\gamma}}

\newcommand{\hol}{{\sf hol}}
\newcommand{\subh}{{\sf subh}}
\newcommand{\harm}{{\sf harm}}
\newcommand{\ph}{{\sf Ph}}
\newcommand{\tl}{\tilde{\lambda}}
\newcommand{\ts}{\tilde{\sigma}}

\def\v{\varphi}
\def\Re{{\sf Re}\,}
\def\Im{{\sf Im}\,}
\def\rg{{\sf rg}\,}
\def\Lr{{\sf Lr}\,}

\def\dist{{\rm dist}}
\def\const{{\rm const}}
\def\rk{{\rm rank\,}}
\def\id{{\sf id}}
\def\aut{{\sf aut}}
\def\Aut{{\sf Aut}}
\def\CR{{\rm CR}}
\def\GL{{\sf GL}}
\def\U{{\sf U}}

\def\la{\langle}
\def\ra{\rangle}

\newtheorem{theorem}{Theorem}[section]
\newtheorem{lemma}[theorem]{Lemma}
\newtheorem{proposition}[theorem]{Proposition}
\newtheorem{corollary}[theorem]{Corollary}

\theoremstyle{definition}
\newtheorem{definition}[theorem]{Definition}
\newtheorem{example}[theorem]{Example}

\theoremstyle{remark}
\newtheorem{remark}[theorem]{Remark}
\numberwithin{equation}{section}

\title[abstract approach to Loewner chains]{An abstract approach to Loewner chains}

\author[L. Arosio]{Leandro Arosio$^{\ddagger}$}
\address{Istituto Nazionale di Alta Matematica ``Francesco Severi'', Citt\`a Universitaria, Piazzale Aldo Moro 5, 00185 Rome, Italy}
\email{arosio@altamatematica.it}
\thanks{$^{\ddagger}$Titolare di una Borsa della Fondazione Roma - Terzo Settore  bandita dall'Istituto Nazionale di Alta Matematica}

\author[F. Bracci]{Filippo Bracci}
\address{Dipartimento di Matematica, Universit\`a Di Roma ``Tor Vergata'',
Via della Ricerca Scientifica 1, 00133 Rome, Italy}
\email{fbracci@mat.uniroma2.it}

\author[H. Hamada]{Hidetaka Hamada$^{*}$}
\address{Faculty of Engineering, Kyushu Sangyo University, 3-1 Matsukadai 2-Chome,
Higashi-ku Fukuoka 813-8503, Japan}
\email{h.hamada@ip.kyusan-u.ac.jp}
\thanks{$^{*}$Partially
supported by Grant-in-Aid for Scientific Research (C)
No. 22540213 from Japan Society for the Promotion of Science, 2011}

\author[G. Kohr]{Gabriela Kohr$^{**}$}
\address{Faculty of Mathematics and Computer Science,
Babe\c{s}-Bolyai University, 1 M. Kog\u{a}l\-niceanu Str., 400084
Cluj-Napoca, Romania} \email{gkohr@math.ubbcluj.ro}
\thanks{$^{**}$Partially
supported by the Romanian Ministry of Education and Research, UEFISCSU-CNCSIS
Grant PN-II-ID 524/2007}


\date{}


\begin{abstract}
We present a new geometric construction of Loewner chains in one
and several complex variables which holds  on complete hyperbolic complex
manifolds and prove that there is essentially a one-to-one
correspondence between evolution families of order $d$ and
Loewner chains of the same order.  As a consequence we obtain an univalent solution $(f_t\colon M\to N)$ for any
Loewner-Kufarev  PDE. The problem of finding  solutions given by univalent mappings $(f_t\colon M\to \C^n)$ is reduced to investigating whether the complex manifold $\cup_{t\geq 0}f_t(M)$ is biholomorphic to a domain in $\C^n$. We apply
such results to the study of univalent mappings from the
unit ball  $\B^n$ to $\C^n$.
\end{abstract}

\maketitle

\section{Introduction}

Loewner's partial differential equation
$$
\frac{\partial f_s}{\partial s}(z)=-\frac{\partial f_s}{\partial z}(z)G(z,s),\quad \mbox{a.e. }s\geq 0, z\in M
$$
received much attention from
mathematicians since Charles Loewner \cite{Loewner}
introduced it in 1923 to study extremal problems and, later,
P.P. Kufarev \cite{Kuf1943} and C. Pommerenke \cite{Po64},  \cite{Po} fully
developed the original theory. Such an
equation was a cornerstone in the de Branges' proof of the
Bieberbach conjecture. In 1999 O. Schramm \cite{Schramm}
introduced a stochastic version of the original differential
equation, nowadays known as SLE, which, among other things, was
a basic tool to prove Mandelbrot's conjecture by himself, G.
Lawler and W. Werner.

Loewner's original theory has been extended (see \cite{Pf74}, \cite{Pf75}, \cite{GHK01},
\cite{GHK09}, \cite{GK03}, \cite{Por91}) to higher
dimensional balls in $\C^n$ and successfully used to study
distortion, star-likeness, spiral-likeness and other geometric
properties of univalent mappings in higher dimensions.

Very recently, the second named author with M. Contreras and S.
D\'iaz-Madrigal \cite{BCD08}, \cite{BCD09} and Contreras,
D\'iaz-Madrigal and P. Gumenyuk \cite{CDG09} proposed a general
setting for the Loewner theory, which works also on complete hyperbolic complex manifolds.
While the classical theory deals with
normalized objects, this general theory does not, and
encloses the classical theory as a special case.

The aim of this paper is to present a general geometric
construction of Loewner chains on complete hyperbolic complex manifolds which
does not use any limit process (and thus it is new also for the
unit disc case) but relies on the apparently new
interpretation of  Loewner chains as the direct limit of evolution families, and to give applications of
such a theory to geometric properties of univalent mappings on the
unit ball. To be more precise, we need  some definitions. In the following, $M$ is a complete hyperbolic complex manifold of dimension $n$, and  $d\in [1,+\infty]$.
An {\sl $L^d$-evolution family}   on $M$ is a family $(\v_{s,t})_{0\leq s\leq t}$ of holomorphic self-mappings of $M$ satisfying the {\sl evolution property}
$$\v_{s,s}=\id,\quad \v_{s,t}=\v_{u,t}\circ \v_{s,u},\quad 0\leq s\leq u\leq t,$$   and
$t\mapsto \v_{s,t}(z)$ has some $L^d_{\sf loc}$-type regularity
locally uniformly with respect to $z\in M$ (see Definition \ref{L^d_EF}).

$L^d$-evolution families are trajectories of certain time-dependent
holomorphic vector fields on $M$, called {\sl Herglotz vector
fields}. An  $L^d$-Herglotz vector field $G(z,t)$  on  $M$ is a weak holomorphic
vector field in the sense of Carath\'eodory which satisfies a
suitable $L^d_{\sf loc}$-bound in $t$ uniformly on compacta of
$M$ and such that for almost every $t\geq 0$ the vector field
$z\mapsto G(z,t)$ is semicomplete (see Definition
\ref{Her-vec-man}).

The main result in \cite{BCD09} states that  there is a one-to-one correspondence between
evolution families  and Herglotz vector fields. The bridge for
such a correspondence is given by the following
Loewner-Kufarev ODE:
\begin{equation}
\label{Lo-ODE}
\frac{\de\v_{s,t}}{\de t}(z)= G(\v_{s,t}(z), t), \quad \hbox{a.e. } t\in \R^+.
\end{equation}

Both classical radial and chordal Loewner ODE in the unit disc
are just particular cases of such an equation (see
\cite{BCD08}).

$L^d$-Evolution families are strictly
related to $L^d$-Loewner chains. Such chains are
defined in the unit disc $\D$ in \cite{CDG09} as families  of univalent mappings $(f_t\colon \D\to\C)_{t\geq 0}$  such
that $f_s(\D)\subseteq f_t(\D)$ for all $0\leq s\leq t$
and satisfying an $L^d_{\sf loc}$ bound in $t$ uniformly on
compacta of $\D$. The classical Loewner chains in the unit disc
are particular cases of such chains.

The correspondence between evolution families and Loewner chains is provided by a functional equation:
an $L^d$-evolution family  and an $L^d$-Loewner chain are {\sl associated} if
\begin{equation}\label{associated}
f_s=f_t\circ \v_{s,t},\quad 0\leq s\leq t.
\end{equation}

In \cite{CDG09} it is proved that given an $L^d$-Loewner chain $(f_t)$
 in the unit disc $\D$, the family
$$(\v_{s,t}:=f_t^{-1}\circ f_s)$$ is an associated $L^d$-evolution family  and, conversely,  any $L^d$-evolution family
 admits  a unique (up to biholomorphisms) associated $L^d$-Loewner chain. Such a result, as already in the classical
theory, is based on a scaling limit process.

Similar results, in the case of $L^\infty$-evolution families  in the unit ball $\B^n\subset \C^n$ fixing the origin
and having a normalized differential at the origin, have been
obtained in \cite{GHK01}, \cite{GK03}.
In such works
Loewner chains are defined as image-increasing sequences of
univalent mappings on the ball with image in $\C^n$ fixing the
origin and having the differential subjected to some
normalization at the origin. Again,  Loewner chains are defined starting from
normalized evolution families by means of a scaling limit
process.

In this paper we propose a definition of
$L^d$-Loewner chains on complete hyperbolic complex manifolds and prove that equation (\ref{associated})
provides  a  one-to-one correspondence (up to biholomorphisms) between
$L^d$-Loewner chains and $L^d$-evolution families.
Since there exist complete hyperbolic complex manifolds (even non-compact ones) which are not biholomorphic to domains in $\mathbb{C}^n$,  requiring each $f_t$ to be a univalent mapping from $M$ to $\mathbb{C}^n$ would be unnecessarily restrictive. Hence we give the following definition:
let $N$ be a complex manifold of the same
dimension of $M$  and let $d_N$
denote the  distance  induced on $N$ by some Hermitian metric. A
family $(f_t: M\to N)_{t\geq 0}$ is an {\sl
$L^d$-Loewner chain}   if
\begin{itemize}
  \item[LC1.] For each $t\geq 0$ fixed, the mapping $f_t: M\to N$
  is univalent,
  \item[LC2.] $f_s(M)\subset f_t(M)$ for all $0\leq s\leq
  t<+\infty$,
  \item[LC3.] For any compact set $K\subset\subset M$ and any
  $T>0$ there exists a $k_{K,T}\in L^d([0,T], \R^+)$ such
  that for all $z\in K$ and for all $0\leq s\leq t\leq T$
  \[
d_N(f_s(z), f_t(z))\leq \int_{s}^t k_{K,T}(\xi)d\xi.
  \]
\end{itemize}

The main results of the present paper can be summarized as
follows.
\begin{theorem}
\label{main-intro} Let $M$ be a complete hyperbolic complex manifold of
dimension $n$. Let $(\v_{s,t})$ be an $L^d$-evolution family on
$M$. Then there exists an associated $L^d$-Loewner chain $(f_t\colon M\to N)$.
If $(g_t\colon M\to Q)$ is another $L^d$-Loewner chain associated with $(\v_{s,t})$, then there exists a
 biholomorphism $$\Lambda\colon \bigcup_{t\geq 0} f_t(M)\to   \bigcup_{t\geq 0} g_t(M)$$ such that $$g_t=\Lambda\circ f_t,\quad t\geq 0.$$

Conversely, if $(f_t\colon M\to N)$ is an $L^d$-Loewner chain, then $(\v_{s,t}:=f_t^{-1}\circ f_s)$ is an associated $L^d$-evolution family.
\end{theorem}

The first part of the result holds more generally on taut
manifolds (see Theorems \ref{algebraic_theorem} and \ref{ev-to-Low}). The second part is proved in Theorem
\ref{Low-to-ev}. In order to prove the result we exploit a {\sl
kernel convergence theorem} on complete hyperbolic complex manifolds which we
prove in Theorem \ref{kernel}.

The associated $L^d$-Loewner chain $(f_t\colon M\to N)$ is constructed as the direct limit of the $L^d$-evolution family $(\v_{s,t})$ in the following way. Define an equivalence relation on the product $M\times \mathbb{R}^+$:
$$(x,s)\sim (y,t)\quad\mbox{iff}\quad   \v_{s,u}(x)=\v_{t,u}(y) \mbox{ for $u$ large enough},$$ and define $N:=(M\times \mathbb{R}^+)/_\sim.$
Let $\pi\colon M\times \mathbb{R}^+\to N$ be the projection on the quotient, and let  $i_t\colon M\to M\times\mathbb{R}^+$ be the injection $i_t(x)=(x,t)$. The chain is then defined as $$f_t:= \pi\circ i_t,\quad t\geq 0.$$
Equation (\ref{associated}) holds since
$$ \pi\circ i_s= \pi\circ i_t \circ \v_{s,t}, \quad 0\leq s\leq t.$$
Then we endow  $N=\bigcup_{t\geq 0} f_t(M)$ with a complex manifold structure which makes the mappings $f_t$'s holomorphic and we prove the $L^d$-estimate.

As a consequence of Theorem \ref{main-intro}, we can define the {\sl Loewner range} ${\sf Lr}(\v_{s,t})$ of  $(\v_{s,t})$ as the biholomorphism class of $\bigcup_{t\geq 0} f_t(M)$, where $(f_t)$ is any associated $L^d$-Loewner chain. The Loewner range can be seen as an analogue of the abstract basin of attraction defined by Fornaess and Stens\o nes in the setting of discrete holomorphic dynamics with an attractive fixed point \cite{Fo-Ste}. This  suggests the following dynamical interpretation of the Loewner range: let $Q$ be a complex manifold and assume that  an algebraic evolution family of automorphisms $(\Phi_{s,t}\colon Q\to Q)$ has an invariant domain $D\subset Q$. Let $(\v_{s,t}\colon D\to D)$ be the algebraic evolution family obtained restricting $(\Phi_{s,t})$. Then the complex manifold $$\{z\in Q : \Phi_{0,t}(z)\in D\mbox{ for $t$ big enough}\}$$ is biholomorphic to $\Lr(\v_{s,t})$.

If $(\v_{s,t})$ is an $L^d$-evolution family on the unit disc $\mathbb{D}$ the Loewner range has to be simply connected and cannot be compact, thus by the uniformization theorem it has to be biholomorphic to $\mathbb{D}$ or $\mathbb{C}$, and, as noticed also in \cite{CDG09}, the choice depends on the dynamics of $(\v_{s,t})$. Generalizing this result we prove that  if $(f_t)$ and $(\v_{s,t})$ are associated, then
$$f_s^*\kappa_{{\sf Lr}(\v_{s,t})}=\lim_{t \to \infty}\v_{s,t}^*\kappa_M,\quad s\geq 0,$$ where $\kappa_M$ and $\kappa_{\sf Lr(\v_{s,t})}$ are the Kobayashi pseudometrics of $M$ and ${\sf Lr}(\v_{s,t})$ respectively. Using results from Fornaess and Sibony \cite{F-S} we  provide in Theorem \ref{forsi} some conditions on the corank of the Kobayashi pseudometric in order
to determine the Loewner range of an $L^d$-evolution family.

In dimension one, Theorem \ref{main-intro} and the
uniformization theorem allow to recover both the classical
results of Loewner,  Kufarev, Pommerenke and the new results by
Contreras, D\'iaz-Madrigal and Gumenyuk. In higher dimensions
these results are new.

Let now  $G(z,t)$ be an $L^d$-Herglotz vector field whose flow is given by an $L^d$-evolution family as in (\ref{Lo-ODE}), and let $N$ be a complex manifold of dimension $n$.
Theorem \ref{main-intro} yields that  a family of univalent mappings $(f_t\colon M\to N)$  solves the Loewner-Kufarev PDE
$$\frac{\partial f_s}{\partial s}(z)=-(df_s)_zG(z,s),\quad \mbox{a.e. }s\geq 0, z\in M$$ if and only if it is an $L^d$-Loewner chain associated with $(\v_{s,t})$. A  solution  given by univalent mappings $(f_t\colon M\to \C^n)$ exists if and only if the Loewner range $\Lr(\v_{t,s})$ is biholomorphic to a domain in $\C^n$.

In Section \ref{conjugacy} we introduce a notion of conjugacy for $L^d$-evolution families which preserves the Loewner range.
In Section \ref{ss62} we give examples of $L^d$-Loewner
chains in the unit ball generated by the Roper-Suffridge
extension operator.
In Section \ref{ss63} we consider
spiral-shaped and star-shaped mappings and give a characterization of
such mappings.

\section{Evolution families and Herglotz vector fields}

In the rest of this paper, unless differently stated, all
manifolds are assumed to be connected. Let $M$ be a complex
manifold and let $d_M$ denote the distance associated with a
given Hermitian metric on $M$. In the sequel we will also use
the  Kobayashi pseudodistance $k_M$ on $M$ and the associated
Kobayashi pseudometric $\kappa_M$ on $M$. For definitions and
properties we refer the reader to the books \cite{abate}, \cite{Kob}.

\begin{definition}\label{L^d_EF}
Let $M$ be a taut manifold.
A family $(\v_{s,t})_{0\leq s\leq t}$ of holomorphic
self-mappings of $M$ is  an {\sl
 evolution family of order $d\geq 1$} (or $L^d$-evolution family) if it satisfies the {\sl evolution property}
\begin{equation}\label{evolution_property}
\v_{s,s}=\id,\quad \v_{s,t}=\v_{u,t}\circ \v_{s,u},\quad 0\leq s\leq u\leq t,
\end{equation}
 and if for any $T>0$ and for any compact set
  $K\subset\subset M$ there exists a  function $c_{T,K}\in
  L^d([0,T],\R^+)$ such that
  \begin{equation}\label{ck-evd}
d_M(\v_{s,t}(z), \v_{s,u}(z))\leq \int_{u}^t
c_{T,K}(\xi)d \xi, \quad z\in K,\  0\leq s\leq u\leq t\leq T.
  \end{equation}
\end{definition}

The following lemma is proved in \cite[Lemma 2]{BCD09}.
\begin{lemma}\label{jointly}
Let $d\in[1,+\infty]$. Let $(\v_{s,t})$ be an $L^d$-evolution family. Let $\Delta :=\{(s,t):0\leq s\leq t\}$. Then the mapping $$(s,t)\mapsto \v_{s,t}$$ from $\Delta$ to $\hol(M,M)$ endowed with the topology of uniform convergence on compacta is jointly continuous. Hence the mapping $\Phi(z,s,t):= \v_{s,t}(z)$ from $M\times \Delta$ to $M$ is jointly continuous.
\end{lemma}

\begin{proposition}\label{ev-univ}
Let $d\in[1,+\infty]$.
Let $(\v_{s,t})$ be an $L^d$-evolution family. Then for all $0\leq s\leq t$ the mapping $(\v_{s,t})$ is univalent.
\end{proposition}
\begin{proof}
We proceed by contradiction. Suppose  there exists $0<s<t$ and $z\neq w$ in $M$ such that $\v_{s,t}(z)=\v_{s,t}(w).$ Set $r:= \inf \{u\in [s,t]: \v_{s,u}(z)=\v_{s,u}(w)\}.$
Since by Lemma \ref{jointly} $\lim_{u\to s+}\v_{s,u}=\id$ uniformly on compacta, we have $r>s$. If $u\in (s,r)$, $$\v_{u,r}(\v_{s,u}(z))=\v_{u,r}(\v_{s,u}(w)),$$ and since $\v_{s,u}(z)\neq \v_{s,u}(w),$ the mappings $\v_{u,r}$, $u\in (s,r)$, are not univalent on a fixed relatively compact subset of $M$. But by Lemma \ref{jointly} $\lim_{u\to r-} \v_{u,r}=\id$ uniformly on compacta, which is a contradiction since the identity mapping is univalent.
\end{proof}

\begin{definition}
\label{Her-vec-man} A \textit{weak holomorphic vector field of
order $d\geq 1$} on $M$ is a mapping $G:M\times \R^+\to
TM$ with the following properties:
\begin{itemize}
\item[(i)] The mapping $G(z,\cdot)$ is measurable on $\R^+$ for all
$z\in M$.
\item[(ii)] The mapping $G(\cdot,t)$ is holomorphic on $M$ for all $t\in \R^+$.
\item[(iii)] For any compact set $K\subset M$ and all $T>0$, there exists
a function $C_{K,T}\in L^d([0,T],\mathbb{R}^+)$ such that
$$\|G(z,t)\|\leq C_{K,T}(t),\quad z\in K, \mbox{ a.e.}\ t\in [0,T].$$
\end{itemize}

Let $M$ be a taut manifold. Assume moreover that the Kobayashi distance $k_M$ satisfies
$k_M\in C^1(M\times M\setminus \mbox{Diag})$.
A \textit{Herglotz vector field of order $d$} is a weak holomorphic vector field of order $d$ such that
\begin{equation}\label{herglotz_def}
(dk_M)_{(z,w)}(G(z,t),G(w,t))\leq 0,\quad z,w\in M, z\neq w,\ \mbox{ a.e. }
t\geq 0.
\end{equation}
\end{definition}
\begin{remark}
If $M$ is complete hyperbolic, then  condition (\ref{herglotz_def})
 means exactly that for almost every $t\geq 0$ the
holomorphic vector field $z\mapsto G(z,t)$ is semicomplete
 (this is proved in \cite{BCD10}  for
strongly convex domains, but the same proof works in the complete hyperbolic case).
\end{remark}

The result in \cite{BCD09} which we will use in the sequel is
the following:
\begin{theorem}\label{prel-thm}
Let $M$ be a complete hyperbolic complex manifold with Kobayashi
distance $k_M$. Assume that $k_M\in C^1(M\times M\setminus
\hbox{Diag})$. Then for any  Herglotz vector field $G$ of order
$d\in [1,+\infty]$ there exists a unique $L^d$-evolution family
$(\v_{s,t})$  over $M$ such that for all
$z\in M$
  \begin{equation}\label{solve}
\frac{\de \v_{s,t}}{\de t}(z)=G(\v_{s,t}(z),t) \quad
\hbox{a.e.\ } t\in [s,+\infty).
  \end{equation}
Conversely for any  $L^\infty$-evolution family $(\v_{s,t})$  over $M$ there exists a  Herglotz vector field $G$ of
order $\infty$ such that \eqref{solve} is satisfied. Moreover,
if $H$ is another weak holomorphic vector field which satisfies
\eqref{solve} then $G(z,t)=H(z,t)$ for all $z\in M$ and almost
every $t\in \R^+$.
\end{theorem}

\section{Kernel convergence  on complex manifolds}

Let $B(z_0,r)\subset \C^n$ denote the Euclidean open ball of center $z_0$
and radius $r>0$ (as customary, we denote by $\B^n:=B(0,1)$ the Euclidean open ball centered at the origin and radius $1$).

\begin{proposition}\label{local1-1}
Let $U\subset \mathbb{C}^n$ be an open set. Let $f_k\colon U\rightarrow \mathbb{C}^n$
be a sequence of univalent mappings.
Assume that $f_k\rightarrow f$ uniformly on compacta and that $f$ is univalent.
Then for all $z_0\in U$ and $0<s<r$ such that $B(z_0,s)\subset\subset B(z_0,r)\subset\subset U$
there exists $m=m(z_0,s,r)$ such that if $k>m$ then  $$ f(B(z_0,s))\subset f_k(B(z_0,r)).$$
\end{proposition}

\begin{proof}
Let $K=f\left(\overline{B(z_0,s)}\right)$, $\gamma=\partial B(z_0,r)$ and $\Gamma=f(\gamma).$ Then $K\cap \Gamma=\varnothing$ since $f$ is univalent on $U$.

Let $\eta$ be the Euclidean distance between $\Gamma$ and $K$. Then $\eta>0$ and $$\eta=\min \{\|f(z)-w\|:\ w\in K,\|z-z_0\|=r\}.$$
If $u_0\in K$ then $\|f(z)-u_0\|\geq \eta$ for $z\in\gamma$, and since $f_k\rightarrow f$ uniformly on $\gamma$ there exists $m>0$ such that if $k\geq m$  and $z\in \gamma$  then  $$\|f(z)-f_k(z)\|<\|f(z)-u_0\|.$$

Rouch\'e theorem in several complex variables (see \cite[Theorem 9.3.4]{L2})  yields then that $f_k(z)-u_0$ and $f(z)-u_0$ have the same number of zeros on $B(z_0,r)$ counting multiplicities. But $f(z)-u_0$ has a zero in $B(z_0,r)$ since $u_0\in K$, and thus $u_0\in f_k(B(z_0,r))$ for $k\geq m$. The constant $m>0$ does not depend on $u_0\in K$, hence we have the result.
\end{proof}
\begin{corollary}\label{corollary_local1-1}
Let $U\subset \mathbb{C}^n$ be an open set.  Let $(f_k)$ be a sequence of univalent mappings $f_k\colon U\rightarrow \mathbb{C}^n$ converging uniformly on compacta to a univalent mapping $f$. Then any compact set $K\subset f(U)$ is eventually contained in $f_k(U)$.
\end{corollary}
\begin{proof}
All the balls $B(z,s)\subset\subset U$ give an open covering of $U$. Since $K$ is compact there is a finite number of balls $B(z_i,s_i)\subset\subset U$  such that $K\subset \bigcup_i f(B(z_i,s_i)),$ hence Proposition \ref{local1-1} yields the result.
\end{proof}

\begin{definition}
Let $(\Omega_k)$ be a sequence of open subsets of a manifold $M$. The  {\sl kernel}   $\Omega$ is the biggest open set such that for all compact sets $K\subset \Omega$ there exists $m=m(K)$ such that if $k\geq m$ then $K\subset \Omega_k$. We say that the sequence $(\Omega_k)$ {\sl kernel converges} to $\Omega$ (denoted $\Omega_k\to \Omega)$ if every subsequence of $(\Omega_k)$ has the same kernel $\Omega$.
\end{definition}

Note that by the very definition the kernel is an open set,
possibly empty. It might be empty as the following example
shows:

\begin{example}
Let $M=\C$ and $f_k:\D\to \C$ defined by $f_k(z)=\frac{1}{k}z$.
Then $(f_k)$ is a sequence of univalent mappings
converging uniformly on compacta to $0$, and $f_k(\D)\to\varnothing.$
\end{example}

We have the following result. Another version of the kernel
convergence theorem in $\mathbb{C}^n$ may be found in
\cite{DGHK}.

\begin{theorem}\label{kernel}[Kernel convergence]
Let $(f_k)$ be a sequence of univalent mappings
from a complete hyperbolic complex manifold $M$ to a complex manifold $N$ of the same dimension.
Suppose that $(f_k)$ converges uniformly on compacta to a univalent mapping $f$. Then $f(M)$ is a connected component of the kernel $\Omega$ of the sequence $(f_k(M))$, and $(f_k^{-1}|_{f(M)})$ converges uniformly on compacta to $f^{-1}|_{f(M)}$. In particular if $\Omega$ is connected then $(f_k(M))\to \Omega.$
\end{theorem}

\begin{proof}
Let $K\subset f(M)$ be a compact set. We want to prove that eventually $K\subset f_k(M)$.
Let $\mathcal{U}=\{U_\alpha\}$  be an open covering of $M$ such that any $U_\alpha$ is biholomorphic to $\mathbb{B}^n$, and let $\mathcal{H}$ be the open covering of $M$ given by all open subsets $H$ satisfying the following property: there exists $U_\alpha\in \mathcal{U} $ such that $H\subset\subset U_\alpha$
(notice that $f(H)$ is then relatively compact in some coordinate chart of $N$).  Note that $f$ is an open mapping since $M$ and $N$ have the same dimension, thus $f_*\mathcal{U}=\{f(U_\alpha)\}_{U_\alpha\in\mathcal{U}}$ is an open covering of $f(M).$

 Since $K$ is compact there exist a finite number  of  open subsets $H_i\in\mathcal{H}$  such that $K\subset \bigcup_i f(H_i)$. Note that  on  $H_i$ the sequence $f_k$ takes eventually values  in some $f(U_{\alpha_i})$ thanks to uniform convergence on compacta.
By using a partition of unity it is easy to see that there exist a finite number of compact sets $K_i$ such that $K_i\subset f(H_i)$ and $K=\bigcup_i K_i$. Thus we can assume $M\subset \mathbb{C}^n$ and $N=\mathbb{C}^n$, and the claim follows from Corollary \ref{corollary_local1-1}.

Thus $f(M)$ is a subset of the kernel $\Omega$ of the sequence $(f_k(M))$. This implies that on  any compact set $K\subset f(M)$ the sequence $f_k^{-1}\colon K\rightarrow M$ is eventually defined. Let $\Omega_0$ be the connected component of the kernel which contains $f(M).$  We want to prove that $(f_k^{-1}|_{\Omega_0})$ admits a subsequence converging uniformly on compacta. Assume that $(f_k^{-1}|_{\Omega_0})$ is compactly divergent. Since $M$ is complete hyperbolic, this is equivalent to assume that for all fixed $z_0\in M$ and compact sets $K\subset{\Omega_0}$ we have
\begin{equation}\label{liminfmin}
\liminf_{k\rightarrow \infty}\left(\min_{w\in K}k_M(f^{-1}_k(w),z_0)\right)=+\infty.
\end{equation}
Let $j\geq 0$ and let $$K(j):=\{f(z_0)\}\cup\bigcup_{k\geq j}\{f_k(z_0)\}.$$ Since $f_k(z_0)\rightarrow f(z_0)$, the set $K(j)$ is compact. Since $f(M)$ is open there exists $m>0$ such that $K(m)\subset f(M)\subset {\Omega_0}$. But $$k_M\left(f_k^{-1}(f_k(z_0)),z_0\right)=0,$$ in contradiction with (\ref{liminfmin}).

Let $(f^{-1}_{k_i}|_{\Omega_0})$ be a converging subsequence and let $g:\Omega_0 \to M$ be its limit.
Let $w_0\in {\Omega_0}$.
The sequence $(f^{-1}_{k_i}(w_0))$ is eventually defined and converging to some $z=g(w_0)$. Thus $w_0=f_{k_i}(f^{-1}_{k_i}(w_0))\rightarrow f(z)$, which implies that ${\Omega_0}=f(M)$ and that $g(w_0)=f^{-1}(w_0)$, hence $(f^{-1}_k|_{\Omega_0})$ converges to $f^{-1}|_{\Omega_0}$ uniformly on compacta.
\end{proof}

The condition that the sets are open is important, as the
following example shows:

\begin{example}
Let $\D:=\{\zeta \in \C : |\zeta|<1\}$. Let $f_k : \D \to \C^2$
be defined by $f_k(\zeta):=(\zeta, \frac{1}{k} \zeta)$.  Then
$(f_k)$ is a sequence of univalent discs which
converges uniformly on compacta to the injective disc
$\zeta\mapsto (\zeta, 0)$.
The only compact set in $\mathbb{C}^2$ which is eventually contained in
$f_k(\D)$  is $\{0\}$.
\end{example}

\section{Loewner chains on complex manifolds}

As we will show in what follows, some properties of Loewner chains are related only to the algebraic properties of evolution family and not to $L^d$ regularity. Hence, it is natural to introduce the following:

\begin{definition}\label{algebraic_EF}
Let $M$ be a complex manifold.
An {\sl algebraic evolution family} is a family $(\v_{s,t})_{0\leq s\leq t}$ of univalent self-mappings of $M$  satisfying the  evolution property (\ref{evolution_property}).
\end{definition}

Thanks to Proposition \ref{ev-univ}, an $L^d$-evolution family is an algebraic evolution family ({\sl i.e.}, it is univalent).

\begin{definition}
Let $M, N$ be two complex manifolds of the same dimension.  A family  of holomorphic mappings $(f_t: M\to N)_{t\geq 0}$ is a {\sl
subordination chain}  if for each $0\leq s\leq t$ there
exists a holomorphic mapping $v_{s,t}:M\to M$ such that
$f_s=f_t\circ v_{s,t}$.
A subordination chain $(f_t)$ and an algebraic evolution family $(\v_{s,t})$ are {\sl associated} if
$$ f_s=f_t\circ \v_{s,t},\quad 0\leq s\leq t.$$

An {\sl algebraic Loewner chain} is a  subordination chain such that each mapping $f_t: M\to N$ is univalent.   The {\sl range}  of an algebraic  Loewner chain is defined as $\rg(f_t):=\bigcup_{t\geq 0}f_t(M).$
An algebraic Loewner chain $(f_t\colon M\to N)$ is {\sl surjective} if $\rg(f_t)=N$.
\end{definition}
\begin{remark}
Equivalently an algebraic Loewner chain can be defined as a  family  of univalent mappings $(f_t: M\to N)_{t\geq 0}$ such that $$f_s(M)\subset f_t(M),\quad 0\leq s\leq t.$$
\end{remark}

\begin{definition}
Let $d\in [1,+\infty]$. Let $M, N$ be two complex manifolds of
the same dimension. Let $d_N$ be the distance induced by a Hermitian metric on $N$. An algebraic Loewner chain $(f_t\colon M\to N)$  is a {\sl Loewner chain of order} $d\in [1,+\infty]$ (or $L^d$-Loewner chain) if for any compact set $K\subset\subset M$ and any
  $T>0$ there exists a $k_{K,T}\in L^d([0,T], \R^+)$ such that
 \begin{equation}\label{LCdef}
d_N(f_s(z), f_t(z))\leq \int_{s}^t k_{K,T}(\xi)d\xi
 \end{equation}
for all $z\in K$ and for all $0\leq s\leq t\leq T$.
\end{definition}

\begin{remark}\label{conti}
By (\ref{LCdef}) the mapping $t\mapsto f_t$ is continuous from $\mathbb{R}^+$ to $\hol(M,N)$. Hence the mapping
 $\Psi\colon  M\times \mathbb{R}^+\rightarrow N$ defined as $\Psi(z,t)=f_t(z)$
is jointly continuous.
\end{remark}

\begin{theorem}
\label{Low-to-ev}
Let $d\in [1,+\infty]$. Let $(f_t\colon M\to N)$ be an $L^d$-Loewner
chain. Assume that $M$ is complete hyperbolic. Let
$$\v_{s,t}:= f_t^{-1} \circ f_s,\quad 0\leq s\leq t.$$
Then for any Hermitian metric on $M$, the family $(\v_{s,t})$ is an $L^d$-evolution family on $M$ associated with $(f_t)$.
\end{theorem}

\begin{proof}
It is clear that $(\v_{s,t})$ is an algebraic evolution
family, so that we only need to prove the $L^d$-estimate.

Let $H$ be a compact subset of $f_t(M)$. Set $$L(H,t):=\sup_{\eta,\zeta\in H,\eta\neq\zeta}\frac{d_M(f_t^{-1}(\zeta),f_t^{-1}(\eta))}{d_N(\zeta,\eta)}.$$
Then $L(H,t) <+\infty$,  since $w\mapsto f_t^{-1}(w)$ is locally Lipschitz.

Given a compact subset $K\subset M $  define  $$K_t:= \bigcup_{s\in[0,t]}f_s(K).$$ The set $K_t$ is a compact subset of $f_t(M)$ by Remark \ref{conti} since $K_t=\Psi(K\times [0,t])$.
Let $T>0$ be fixed.
We claim that the function $L(K_t,t)$ on $0\leq t\leq T$ is bounded by a constant $C(K,T)<+\infty$. Assume  that  $L(K_t,t)$ is unbounded. Then there exists a sequence $(t_n)\subset [0,T]$, which we might assume converging to some $t\in [0,T]$, such that $$L(K_{t_n},t_n)\geq n+1,\quad \forall n\geq 0.$$
Hence for any $n\geq 0$ there exist $\zeta_n,\eta_n\in K_{t_n}$ such that $\zeta_n\neq\eta_n$ and
\begin{equation}\label{incr}
\frac{d_M(f_{t_n}^{-1}(\zeta_{n}),f_{t_n}^{-1}(\eta_{n}))}{d_N(\zeta_{n},\eta_{n})}\geq n.
\end{equation}
By passing to a subsequence we may assume that
$\zeta_n \to \zeta\in K_t$ and $\eta_n \to \eta \in K_t$.
By Theorem \ref{kernel}, $f_{t_n}^{-1}\to f_t^{-1}$ uniformly on a neighborhood of $K_t$.
By (\ref{incr}) we have $\eta=\zeta$, since otherwise $$\frac{d_M(f_{t_n}^{-1}(\zeta_{n}),f_{t_n}^{-1}(\eta_{n}))}{d_N(\zeta_{n},\eta_{n})}\rightarrow \frac{d_M(f_t^{-1}(\zeta),f_t^{-1}(\eta))}{d_N(\zeta,\eta)}.$$

Let $U,V$ be two open subsets of $f_t(M)$, both biholomorphic to $\mathbb{B}^n$ such that $\zeta\in U\subset\subset V\subset\subset f_t(M).$
Since by Theorem \ref{kernel} the sequence $(f_{t_n}^{-1})$ converges to $f_t^{-1}$ uniformly on $V$, we have that eventually $ f_{t_n}^{-1}(U)\subset f^{-1}_t(V)$. The sequence $(f_{t_n}^{-1}|_U)$ is thus equibounded and by Cauchy estimates it is equi-Lipschitz in a neighborhood of $\zeta$, which  contradicts (\ref{incr}) and thus proves the claim.

Let $\Delta_T:=\{(s,t):0\leq s\leq t\leq T\}$. Then the mapping $(s,t)\mapsto f^{-1}_t\circ f_s$ from $\Delta_T$ to $\hol(M,M)$ endowed with the topology of uniform convergence on compacta is continuous. Indeed, let $(s_n,t_n)\rightarrow (s,t)$. Let $K\subset M $ be a compact set. By Remark \ref{conti} the set $$K(j):=  f_s(K)\cup\bigcup_{n\geq j}f_{s_n}(K)=\Psi (K,\{s\}\cup\bigcup_{n\geq j}\{s_n\})$$ is compact.  There exists $m>0$ such that $K(m)\subset f_t(M)$. By Theorem \ref{kernel} the sequence $(f_{t_n}^{-1})$ converges to $f_t^{-1}$ uniformly on $K(m)$. This proves that  $(s,t)\mapsto f^{-1}_t\circ f_s$ is continuous.

This implies that the mapping $\Phi\colon M\times \Delta_T\rightarrow M$ defined as $\Phi(z,s,t):=\v_{s,t}(z)$ is jointly continuous. Hence if $K\subset M$ is a compact set,  $$\hat{K}:= \bigcup_{0\leq a\leq b\leq T} \v_{a,b}(K)=\bigcup_{0\leq a\leq b\leq T} f_b^{-1}(f_a(K))$$ is compact in $M$. Therefore, since $$d_M(\v_{s,u}(z),\v_{s,t}(z))=d_M(\v_{s,u}(z),\v_{u,t}(\v_{s,u}(z))),$$ in order to estimate $d_M(\v_{s,u}(z),\v_{s,t}(z))$ for $z\in K$ and $0\leq s\leq u\leq t\leq T$ it is enough to estimate $d_M(\zeta,\v_{u,t}(\zeta))$ for $\zeta\in \hat{K}.$

Since
\begin{align}
d_M(\zeta,\v_{u,t}(\zeta))&=d_M(f_t^{-1}(f_t(\zeta)),f_t^{-1}(f_u(\zeta)))\leq L(\hat{K}_t,t)d_N(f_t(\zeta),f_u(\zeta)) \nonumber\\
&\leq C(\hat{K},T)d_N(f_t(\zeta),f_u(\zeta))\leq C(\hat{K},T)\int_u^tk_{\hat{K},T}(\xi)d\xi,\nonumber
\end{align}
we  are done.
\end{proof}

\begin{theorem}\label{algebraic_theorem}
Any  algebraic evolution family $(\v_{s,t})$ admits  an associated algebraic Loewner chain $(f_t\colon M\to N)$.
Moreover if $(g_t\colon M\to Q)$ is a subordination chain associated with $(\v_{s,t})$ then there exist a holomorphic mapping $\Lambda\colon \rg(f_t)\to Q$ such that $$g_t=\Lambda\circ f_t,\quad \forall t\geq 0.$$ The mapping $\Lambda$ is univalent if and only if $(g_t)$ is an algebraic Loewner chain, and in that case $\rg(g_t)=\Lambda(\rg(f_t)).$
\end{theorem}

\begin{proof}
Define an equivalence relation on the product $M\times \mathbb{R}^+$:
$$(x,s)\sim (y,t)\quad\mbox{iff}\quad   \v_{s,u}(x)=\v_{t,u}(y) \mbox{ for $u$ large enough}.$$
and define $N:=(M\times \mathbb{R}^+)/_\sim$.
Let $\pi\colon M\times \mathbb{R}^+\to N$ be the projection on the quotient, and let  $i_t\colon M\to M\times \mathbb{R}^+$ be the injection $i_t(x)=(x,t)$. Define a family of mappings $(f_t\colon M\to N)$ as $$f_t:= \pi\circ i_t,\quad t\geq 0.$$ Each mapping $f_t$ is injective since ${\pi}|_{M\times \{t\}}$ is injective, and by construction the family $(f_t)$ satisfies$$f_s=f_t\circ\v_{s,t}, \quad 0\leq s\leq t.$$
Thus we have  $f_s(M)\subset f_t(M)$ for $0\leq s\leq t$ and $N=\bigcup_{t\geq 0}f_t(M)$.

Endow the product  $M\times \mathbb{R}^+$ with the product topology, considering on $\R^+$ the discrete topology. Endow $N$ with the quotient topology. Each mapping $f_t$ is continuous and open, hence it is an homeomorphism onto its image. This shows that $N$ is arcwise-connected and Hausdorff since each $f_t(M)$ is arcwise-connected and Hausdorff. Moreover $N$ is second countable since $N=\bigcup_{k\in \mathbb{N}}f_k(M)$.
Now define a complex structure on $N$ by considering the $M$-valued charts $(f^{-1}_t,f_t(M))$ for all $t\geq 0$. This charts are compatible since $f_t^{-1}\circ f_s=\v_{s,t}$ which is holomorphic. Hence  the family $(f_t)$ is an algebraic Loewner chain associate with $(\v_{s,t})$.

If $(g_t\colon M\to Q)$ is a subordination chain associated with $(\v_{s,t})$, then the map $\Psi\colon M\times \mathbb{R}^+\to Q$
$$(z,t)\mapsto g_t(z)$$  is compatible with the equivalence relation $\sim$. The map $\Psi$ passes thus to the quotient defining a holomorphic mapping $\Lambda\colon N\to Q$ such that $$g_t=\Lambda\circ f_t,\quad t\geq 0.$$ The last statement is easy to check.
\end{proof}
As a corollary we have the following.
\begin{corollary}\label{rangebihol}
Let $(\v_{s,t})$ be an algebraic evolution family on a  complex manifold
$M$. Also let $(f_t\colon M\to N)$ and $(g_t\colon M\to Q)$ be two algebraic Loewner chains
associated with $(\v_{s,t})$. Then there exists a
biholomorphism $\Lambda\colon \rg(f_t)\to \rg(g_t)$ such that $g_t=\Lambda \circ
f_t$  for all $t\geq 0$.
\end{corollary}

Thus there exists essentially one algebraic Loewner chain associated with an algebraic evolution family.
\begin{definition}
Let $(\v_{s,t})$ be an algebraic evolution family. By Corollary \ref{rangebihol} the biholomorphism class of the range of an associated algebraic Loewner chain is uniquely determined. We call this class the {\sl Loewner range} of $(\v_{s,t})$ and we denote it by $\Lr(\v_{s,t})$.
\end{definition}

\begin{theorem}\label{ev-to-Low}
Let $d\in [1,+\infty]$.  Let $(\v_{s,t})$ be an $L^d$-evolution family on a taut  manifold $M$, and let $(f_t\colon M\to N)$ be an  associated algebraic Loewner chain. Then  $(f_t)$ is an $L^d$-Loewner chain for any Hermitian metric on $N$.
\end{theorem}
\begin{proof}
Let $K\subset M$ be a compact set. Let $T>0$ be fixed.
By Lemma \ref{jointly} the subset of $M$ defined as $$\hat{K}:=\bigcup_{0\leq s\leq t\leq T} \v_{s,t}(K)$$ is compact. Indeed
$\hat{K}=\Psi(K\times \Delta_T)$, where $\Delta_T=\{(s,t): 0\leq s\leq t\leq T\}$.

Since $f_T$  is locally Lipschitz there exists $C= C(\hat{K})
>0$ such that $$d_N(f_T(z),f_T(w))\leq Cd_M(z,w),\quad z,w\in \hat{K}.$$
The family $(\v_{t,T})_{0\leq t\leq T}$ is equi-Lipschitz on $\hat{K}$, that is there exists $L(\hat{K},T)>0$ such that
\begin{equation}\label{estimatelip}
d_M(\v_{t,T}(z),\v_{t,T}(w))\leq L(\hat{K},T)d_M(z,w),\quad  z,w\in\hat{K},\ t\in[0,T].
\end{equation}
Indeed assume by contradiction that there exist sequences $(z_n),(w_n)$ in $\hat{K}$, and $(t_n)$ in $[0,T]$ such that
\begin{equation}\label{incr2}
\frac{d_M(\v_{t_n,T}(z_{n}),\v_{t_n,T}(w_n))}{d_M(z_{n},w_{n})}\geq n.
\end{equation} By passing to subsequences we can assume $t_n\to t$, $z_n\to z$ and $w_n\to w$, and by (\ref{incr2}) it is easy to see that $z=w$.

Let $U,V$ be two open subsets of $M$, both biholomorphic to $\mathbb{B}^n$ such that $z\in V\subset\subset U\subset\subset M.$
Since  the sequence $(\v_{t_n,T})$ converges to $\v_{t,T}$ uniformly on $U$ we have that eventually $ \v_{t_n,T}(V)\subset \v_{t,T}(U)$. The sequence $(\v_{t_n,T}|_V)$ is thus equibounded and by Cauchy estimates it is equi-Lipschitz in a neighborhood of $z$, which contradicts (\ref{incr2}).

Hence, for all $z\in K$ and $0\leq s\leq t\leq T$ we have
\begin{equation*}
\begin{split}
d_N(f_s(z), f_t(z))&=d_N(f_T(\v_{s,T}(z)), f_T(\v_{t,T}(z)))\leq C d_M(\v_{s,T}(z),
\v_{t,T}(z))
\\ & = C d_M(\v_{t,T}(\v_{s,t}(z)), \v_{t,T}(z))\leq CL(\hat{K},T)d_M(\v_{s,t}(z), z)
\\& = CL(\hat{K},T) d_M(\v_{s,t}(z), \v_{s,s}(z))\leq CL(\hat{K},T) \int_s^t c_{K,T}(\xi) d\xi,
\end{split}
\end{equation*}
by (\ref{ck-evd}). This concludes the proof.
\end{proof}
\begin{corollary}\label{L^dL^d}
Assume that the algebraic evolution family $(\v_{s,t})$ on a complete hyperbolic complex manifold $M$ is associated with the algebraic Loewner chain $(f_t\colon M\to N)$. Then $(\v_{s,t})$ is an $L^d$-evolution family if and only if $(f_s)$ is an $L^d$-Loewner chain.
\end{corollary}
\begin{proof}
It follows from Theorems \ref{ev-to-Low} and \ref{Low-to-ev}.
\end{proof}

When dealing with evolution families defined on a domain $D$
of a complex manifold $N$, a natural question is whether there
exists an associated Loewner chain whose range is contained
in $N$, or, in other terms, whether the Loewner range  is biholomorphic to
a domain of  $N$. This question makes
particularly sense if $D=\B^n$ and $N=\C^n$. In other words:

\medskip

{\bf Open question:} Given an $L^d$-evolution family  on the unit ball $\B^n$ does there exist an
associated $L^d$-Loewner chain with range in $\C^n$?

\medskip

\begin{remark}\label{not_admit}
There exists an algebraic evolution family $(\v_{s,t})$ on $\B^3$ which does not admit any associated algebraic Loewner chain with range in $\C^3$. This follows from  \cite[Section 9.4]{A10}.
\end{remark}

There are several works in this direction, answering such a
question in some normalized class of evolution families
(see \cite{A10}, \cite{GHK01}, \cite{GHK09}, \cite{GHKK08}, \cite{GK03},
\cite{Por91}, \cite{Vo})
but in its generality the question is
still open. Here we give some answers based on the
asymptotic behavior of the Kobayashi pseudometric
under the corresponding evolution family.

\begin{definition}
Let $(\v_{s,t})$ be an algebraic evolution family on a complex manifold $M$. Let $\kappa_M: TM\to \R^+$ be the Kobayashi pseudometric of $M$. For $v\in T_zM$ and $s\geq 0$  we define
\begin{equation}\label{beta}
\beta^s_v(z):=\lim_{t\to \infty} \kappa_M(\v_{s,t}(z); (d\v_{s,t})_z(v)).
\end{equation}
\end{definition}

\begin{remark}
Let $0\leq s\leq u\leq t$. Since the Kobayashi pseudometric is
contracted by holomorphic mappings it follows
\begin{equation*}
\begin{split}
\kappa_M(\v_{s,t}(z); (d\v_{s,t})_z(v))&=\kappa_M(\v_{u,t}(\v_{s,u}(z));
(d\v_{u,t})_{\v_{s,u}(z)}(d\v_{s,u})_z(v))
\\&\leq \kappa_M(\v_{s,u}(z); (d\v_{s,u})_z(v)),
\end{split}
\end{equation*}
hence the limit in \eqref{beta} is well defined.
\end{remark}

\begin{proposition}\label{beta-ev}
Let $(\v_{s,t})$ be
an algebraic evolution family  on a complex manifold $M$. Let
$(f_t\colon M\to N)$ be an associated surjective algebraic Loewner chain.
Then for all $z\in M$ and $v\in T_zM$ it follows
\[
f_s^*\kappa_N(z;v)=\beta^s_v(z).
\]
\end{proposition}

\begin{proof}
Since the chain $(f_t\colon M\to N)$ is surjective, the range $N$ is the union of the growing sequence of complex manifolds $(f_j(M))_{j\in\mathbb{N}}$, thus
$$\kappa_N(f_s(z);(df_s)_z)(v))=\lim_{j\to\infty}\kappa_{f_j(M)}(f_s(z);(df_s)_z(v)).$$
The result follows from
\begin{align}
\lim_{j\to \infty}\kappa_{f_j(M)}(f_s(z);(df_s)_z)(v))&=\lim_{j\to \infty}\kappa_{M}(f_j^{-1}(f_s(z));(df_j^{-1})_{f_s(z)}\circ (df_s)_z(v))\\\nonumber
&=\lim_{j\to\infty}\kappa_M(\v_{s,j}(z);(d\v_{s,j})_z(v))\nonumber.
\end{align}
\end{proof}

As corollaries we find (cf. \cite[Theorem 1.6]{CDG09})
\begin{corollary}\label{Lrdisc}
Let $(\v_{s,t})$ be an algebraic evolution family on the unit disc $\mathbb{D}.$
If there exist  $z\in \D$,\ $v\in T_z\mathbb{D}$,\ $s\geq 0$, such that $\beta_v^s(z)=0$ then
\begin{itemize}
\item[i)] $\Lr(\v_{s,t})$ is biholomorphic to $\mathbb{C}$,
\item[ii)] $\beta_v^s(z)=0$ for all  $z\in \D$,\ $v\in T_z\mathbb{D}$,\ $s\geq 0$.
\end{itemize}
If there exist  $z\in \D$,\ $v\in T_z\mathbb{D}$,\ $s\geq 0$, such that $\beta_v^s(z)\neq 0$ then
\begin{itemize}
\item[i)] $\Lr(\v_{s,t})$ is biholomorphic to $\mathbb{D}$,
\item[ii)] $\beta_v^s(z)\neq0$ for all  $z\in \D$,\ $v\in T_z\mathbb{D}$,\ $s\geq 0$.
\end{itemize}
\end{corollary}
\begin{proof}
Since the Loewner range $\Lr(\v_{s,t})$ is non-compact and simply connected, by the uniformization theorem it has to be biholomorphic to $\mathbb{D}$ or $\mathbb{C}$. Since
$$\kappa_\mathbb{C}(z;v)=0,\quad z\in \mathbb{C},v\in T_z\mathbb{C},$$
$$\kappa_\mathbb{D}(z;v)\neq 0,\quad  z\in \mathbb{D},v\in T_z\mathbb{D},$$
the result follows from Proposition \ref{beta-ev}.
\end{proof}
\begin{corollary}\label{evdisc}
Let $d\in [1,+\infty]$.
Let $(\v_{s,t})$ be an $L^d$-evolution family  on the unit disc $\D$. Then there exists an  $L^d$-Loewner
chain $(f_t)$  associated with $(\v_{s,t})$ such that $\rg(f_t)$ is either the unit disc $\D$ or the complex plane
$\C$.
\end{corollary}
\begin{proof}
It follows from Corollary \ref{Lrdisc} and Theorem \ref{ev-to-Low}.
\end{proof}

Such a result can be generalized in higher
dimension as follows.
 As customary, let us denote by ${\sf aut}(M)$ the group of
holomorphic automorphisms of a complex manifold $M$. Notice that
condition $M$   hyperbolic  and  $M/\sf{aut}(M)$ compact
implies that $M$ is complete hyperbolic (see \cite{F-S}).

\begin{theorem}\label{forsi}
Let $M$ be a hyperbolic complex manifold and assume that
$M/\sf{aut}(M)$ is compact. Let $(\v_{s,t})$ be an algebraic evolution
family on $M$.  Then
\begin{enumerate}
\item If there exists $z\in M$, $s\geq0$ such that $\beta^s_v(z)\neq 0$ for
all $v\in T_zM$ with $v\neq 0$ then  $\Lr(\v_{s,t})$ is biholomorphic to $M$.
\item If there exists $z\in M$, $s\geq0$ such that $\dim_\C\{v\in T_zM: \beta_v^s(z)= 0\}=1$  then  $\Lr(\v_{s,t})$ is a fiber bundle with fiber $\mathbb{C}$
over a closed complex submanifold of $M$.
\end{enumerate}
\end{theorem}

\begin{proof}
It follows at once from Proposition \ref{beta-ev} and
\cite[Theorem 3.2, Main Theorem]{F-S}.
\end{proof}

In particular the previous result applies to $M=\B^n$ (or even to
the polydiscs in $\C^n$) and we obtain

\begin{corollary} \label{image-Cn}
Let $(\v_{s,t})$ be an algebraic evolution family  on the unit ball $\B^n$.
If for some $z\in \B^n$, $s\geq 0$
it follows that $\dim_\C\{v\in \C^n : \beta^s_v(z)=0\}\leq 1$,
then there exists an algebraic Loewner chain $(f_t\colon M\to\mathbb{C}^n)$ associated with $(\v_{s,t})$.
\end{corollary}

\begin{proof}
If the dimension is zero, then by Theorem
\ref{forsi} the Loewner range is biholomorphic to $\B^n\subset \C^n$. If the dimension is one, then by Theorem
\ref{forsi} the Loewner range  is  a fiber bundle with fiber $\mathbb{C}$ over a closed complex submanifold of $\B^n$ and
by \cite[Corollary 4.8]{F-S} it is actually biholomorphic to
$\B^{n-1}\times \C\subset \C^n$.
\end{proof}

If $\dim_\C\{v\in \C^n : \beta^s_v(z)=0\}\geq 2$  the complex structure of the Loewner range can be more complicated: the Loewner range of the algebraic evolution family recalled in Remark \ref{not_admit} has  $\dim_\C\{v\in \C^n : \beta^s_v(z)=0\}= 2$  and is not biholomorphic to a domain of $\C^3$.

\begin{example}
Let $(\v_{s,t})$ be an algebraic evolution family of $\B^2$ such that
$\v_{s,t}(0)=0$ for all $0\leq s\leq t$ and $(d\v_{s,t})_0=
e^{A(t-s)}$ where $A$ is a diagonal matrix with eigenvalues
$i\theta$, $\theta\in \R$ and $\lambda\in \C$ for some $\Re
\lambda\leq 0$. Then $\dim_\C\ker \beta_v^s(0)\leq 1$ (it is
either $1$ if $\Re \lambda<0$ or $0$ if $\Re \lambda=0$ which
is the case if and only if $\v_{s,t}$ are automorphisms).
Therefore in such a case there exists an algebraic Loewner chain with
range in $\C^2$.
\end{example}

The previous example can be generalized as follows:

\begin{example}
Let $G(z,t)$ be an $L^\infty$-Herglotz vector field in $\B^n$
such that $G(0,t)\equiv 0$ and $(d_zG)_{z=0}(\cdot,t)=A(t)$
where $A(t)$ is a diagonal $n\times n$ matrix with eigenvalues
$\lambda_1(t), \ldots, \lambda_n(t)$ where $\lambda_j: \R^+\to
\C$ are functions of class $L^\infty$ such that $\Re
\lambda_j(t)\leq 0$ for almost every $t\geq 0$ and $j=1,\ldots,
n$. Assume that there exists $C>0$ such that
\[
\int_0^t \Re \lambda_j(\tau) d\tau \geq -C, \quad t\geq0,\ j=1,\ldots, n-1.
\]
Let $(\v_{s,t})$ be the associated $L^\infty$-evolution family
in $\B^n$. Then $\v_{s,t}(0)=0$ and $(d\v_{s,t})_0$ is the
diagonal matrix with eigenvalues $\exp \left(\int_s^t
\lambda_j(\tau)d\tau \right)$ for $j=1,\ldots, n$. Hence
$\dim_\C\ker \beta_v^s(0)\leq 1$ and there exists an associated
$L^\infty$-Loewner chain with range in $\C^n$.
\end{example}

\section{Loewner-Kufarev  PDE}\label{PDE}
In this section we prove that $L^d$-Loewner chains  on complete hyperbolic complex manifolds are  the univalent solutions of the Loewner-Kufarev partial differential equation, as in the classical theory of Loewner chains on the unit
ball $\B^n$ in $\mathbb{C}^n$ (see \cite{GHK01}, \cite{GK03}). Other results related to the solutions of the Loewner-Kufarev PDE on $\B^n$ may be
found in \cite{DGHK}.

\begin{proposition}
\label{vitali}
Let $M$ be a complete hyperbolic complex manifold, and let $(f_t\colon M\to N)$ be a Loewner chain of order $d\in[1,+\infty]$ on $M$. Then
there exists a set $E\subset\R^+$ $($independent of $z)$ of
zero measure such that for every  $s\in (0,+\infty)\setminus E$, the
mapping $$M\ni z \mapsto \frac{\partial f_s}{\partial s}(z)\in T_{f_s(z)}N$$
is  well-defined and holomorphic  on $M$.
\end{proposition}
\begin{proof}
The manifold  $M\times (0,+\infty)$ has a countable basis $\mathfrak{B}$ given by products of the type $B\times I$, where $B\subset M$ is an open subset biholomorphic to a ball, and $I\subset (0,+\infty)$ is a bounded open interval.
Let $\mathcal{V}$ be a countable covering of $N$ by charts. By Remark \ref{conti} the mapping $\Psi\colon M\times (0,+\infty)\to N$ is continuous, hence there exists a covering $\mathcal{U}\subset \mathfrak{B}$ of $M\times (0,+\infty)$   such that for any $U\in\mathcal{U}$ there exists $V\in\mathcal{V}$ such that $U\subset \Psi^{-1}(V)$.

We will prove that for each $U=B\times I\in \mathcal{U}$ there exists a set of full measure
$I'\subseteq I$
such that $B\ni z \mapsto \frac{\partial f_s}{\partial s}(z)$
is well defined and holomorphic for all $s\in I'$.
Hence $M \ni z \mapsto \frac{\partial f_s}{\partial s}(z)$ will be well defined
and holomorphic for $s\in \mathbb{R}^{+}\setminus \bigcup_{\mathcal{U}}({I\setminus I'})$ which is a set of full measure in $\mathbb{R}^+$.

We can assume that $M=\B^n$, $N=\C^n$, and that the distance $d_N$ is the Euclidean distance. Since $t\mapsto f_t(z)$ is locally absolutely continuous on $\mathbb{R}^+$ locally
uniformly with respect to $z\in \B^n$, we deduce that for each $z\in \B^n$, there is a null
set $E_1(z)\subset I$ such that for each $t\in I\setminus E_1(z)$, there
exists the limit
\begin{equation*}
\frac{\partial f_t}{\partial t}(z)=\lim_{h\to 0}\frac{f_{t+h}(z)-f_t(z)}{h}.
\end{equation*}

By definition there exists a function $p_k\in L^d(I,
\mathbb{R}^+)$ such that
\begin{equation}\label{pbound}
\|f_s(z)-f_t(z)\|\leq \int_s^tp_k(\xi)d\xi,\quad z\in \B^n_{1-1/k},\quad
 s\leq t\in I.
\end{equation}
Also, since $p_k\in L^d(I, \mathbb{R}^+)$,
we may find a null set $E_2(k)\subset I$ such that
for each $t\in I\setminus E_2(k)$, there exists the limit
\begin{equation}\label{pbound2}
p_k(t)=\lim_{h\to 0}\frac{1}{h}\int_t^{t+h}p_k(\xi)d\xi,\quad k\in\mathbb{N}.
\end{equation}
Next, let $Q$ be a countable set of uniqueness for the holomorphic functions
on $\B^n$ and let
\[E=\bigg(\bigcup_{q\in Q}E_1\left(q\right)\bigg)\bigcup
\bigg(\bigcup_{k=1}^{\infty}E_2(k)\bigg).\]
Then $E$ is a null subset of $\mathbb{R}^+$,
which does not depend on $z\in \B^n$.
Arguing as in the proof of \cite[Theorem 4.1(1)(a)]{CDG09},
it is not difficult to see that \eqref{pbound} and \eqref{pbound2}
imply that for each $s\in I\setminus E$,
the family $$\{(f_{s+h}(z)-f_s(z))/h,\quad 0<|h|<\textsf{dist}(s, \partial I)\}$$ is relatively
compact and has a unique accumulation point for $|h|\to 0$
by Vitali Theorem in several complex variables, proving the result.
\end{proof}

\begin{theorem}\label{LK-PDE}
Let $M$ be a complete hyperbolic complex manifold such that  the Kobayashi
distance $k_M\in C^1(M\times M\setminus \hbox{Diag})$, and let $N$ be a complex manifold of the same dimension.
Let $G:M\times \R^+\to TM$ be a Herglotz vector
field of order $d\in[1,+\infty]$
associated with the $L^d$-evolution family $(\varphi_{s,t})$. Then a family of univalent mappings $(f_t\colon M\to N)$ is an $L^d$-Loewner chain associated with $(\v_{s,t})$ if and only if   it is locally absolutely continuous
on $\R^+$ locally uniformly with respect to $z\in M$ and solves the  Loewner-Kufarev PDE
\begin{equation}\label{equationLK-PDE}
\frac{\partial f_s}{\partial s}(z)=-(df_s)_zG(z,s),\quad\mbox{a.e. }s\geq 0,z\in M.
\end{equation}
\end{theorem}
\begin{proof}
Since $G(z,t)$ and $(\v_{s,t})$ are associated there exists a null set $E_1\subset \R^+$ such that  for all $s\geq 0$, for all $t\in [s,+\infty)\setminus E_1$ and for all $z\in M$,
\[
\frac{\partial \varphi_{s,t}}{\partial t}(z)=G(\varphi_{s,t}(z),t).
\]

Let now $(f_t)$ be an $L^d$-Loewner chain associated with $(\v_{s,t})$.
By Proposition \ref{vitali}, there is a null set $E_2\subset
\R^+$ such that $z\mapsto \frac{\partial f_s}{\partial s}(z)$
is well defined and holomorphic for all $s\in (0,+\infty)\setminus E_2$.
The set $E=E_1\cup E_2$ has also zero measure. It is clear that the mapping
$$t\mapsto L_t(z):=f_t(\varphi_{0,t}(z))$$ is locally absolutely continuous on $\R^+$
locally uniformly with respect to $z\in M$, in view of the conditions (\ref{LCdef}) and
(\ref{ck-evd}). Also $L_t(z)=f_0(z)$ for $z\in M$. Differentiating the last equality with respect to
$t\in (0,+\infty)\setminus E$ we obtain
\begin{equation*}
\begin{split}
0&=(df_t)_{\varphi_{0,t}(z)}\frac{\partial \varphi_{0,t}}{\partial t}(z)+
\frac{\partial f_t}{\partial t}(\varphi_{0,t}(z))\\&=(df_t)_{\varphi_{0,t}(z)}G(\varphi_{0,t}(z),t)+
\frac{\partial f_t}{\partial t}(\varphi_{0,t}(z)),
\end{split}
\end{equation*}
for all $t\in (0,+\infty)\setminus E$ and for all $z\in M$.
Hence
$$\frac{\partial f_t}{\partial t}(w)=
-(df_t)_{w}G(w,t),$$
for all $w$ in the open set $\varphi_{0,t}(M)$ and for all $t\in (0,+\infty)\setminus E$.
The identity theorem for holomorphic mappings provides the result.
\\
To prove the converse, fix $s\geq 0$ and let
$$L_t(z):=f_t(\varphi_{s,t}(z))$$ for $t\in [s,+\infty)$ and $z\in M$.
In view of the hypothesis, it is not difficult to deduce that
$$\frac{\partial L_t}{\partial t}(z)=0,\quad \mbox{ a.e. }\ t\geq 0,\quad
\forall z\in M.$$
Hence $L_t(z)\equiv L_s(z)$, i.e. $f_t(\varphi_{s,t}(z))=f_s(z)$
for all $z\in M$ and $0\leq s\leq t$, which means that $(f_t)$ is an algebraic Loewner chain associated with $(\v_{s,t})$.
Hence  $(f_t)$ is an $L^d$-Loewner chain by Theorem \ref{ev-to-Low}.
\end{proof}

\begin{corollary}\label{corollaryLK-PDE}
Let $M$ be a complete hyperbolic complex manifold such that  the Kobayashi
distance $k_M\in C^1(M\times M\setminus \hbox{Diag})$, and let $N$ be a complex manifold of the same dimension. Every $L^\infty$-Loewner chain $(f_t\colon M\to N)$   solves a Loewner-Kufarev PDE.
\end{corollary}
\begin{proof}
By Theorem \ref{Low-to-ev} there exists an $L^\infty$-evolution
family $(\varphi_{s,t})$  associated with
$(f_t)$. By Theorem \ref{prel-thm} there exists a Herglotz
vector field $G(z,t)$ associated with $(\v_{s,t})$. Theorem \ref{LK-PDE} yields then that the family $(f_t\colon M\to N)$ satisfies
$$\frac{\partial f_s}{\partial s}(z)=-(df_s)_zG(z,s),\quad\mbox{a.e. }s\geq 0,z\in M.$$
\end{proof}

From Theorems \ref{algebraic_theorem} and \ref{ev-to-Low} we easily obtain the following corollary. Let $G(z,t)$ be an $L^d$-Herglotz vector field associated with the $L^d$-evolution family $(\v_{s,t})$.
\begin{corollary}
Let $M$ be a complete hyperbolic complex manifold of dimension $n$ such that  the Kobayashi
distance $k_M\in C^1(M\times M\setminus \hbox{Diag})$. The Loewner-Kufarev PDE (\ref{equationLK-PDE}) admits a solution given by univalent mappings $(f_t\colon M\to N)$ where $N$  is a complex manifold $N$ of dimension $n$. Any other solution with values in a complex manifold $Q$ is of the form $(\Lambda\circ f_t)$ where $\Lambda\colon \rg(f_t)\to Q$ is  holomorphic. Hence a solution given by univalent mappings $(h_t\colon M\to \C^n)$ exists   if and only if the Loewner range $\Lr(\v_{s,t})$ is biholomorphic to a domain in $\C^n$.
\end{corollary}
\section{Conjugacy}\label{conjugacy}
We introduce a notion of conjugacy for $L^d$-evolution families which preserves the Loewner range.  This can be used to put an $L^d$-evolution family in some normal form without changing its Loewner range (cf. \cite[Proposition 2.9]{CDG09}).
\begin{definition}
Let $d\in[1,+\infty].$
Let $(h_t\colon M \to Q)$ be an $L^d$-Loewner chain such that each $h_t\colon M\to Q$ is a biholomorphism. We call $(h_t\colon M\to Q)$ a family of {\sl intertwining mappings of order d}. If $(\v_{s,t}),(\psi_{s,t})$  are $L^d$-evolution families on $M,Q$ respectively and  $$\psi_{s,t}\circ h_s =h_t\circ \v_{s,t},\quad 0\leq s\leq t,$$ then we say that   $(\v_{s,t})$ and $(\psi_{s,t})$ are {\sl conjugate}. It is easy to see that conjugacy is an equivalence relation.
\end{definition}
\begin{lemma}\label{intercomp}
Let $(h_t\colon M\to Q)$ be a family of  intertwining mappings of order $d$ and let $(f_t \colon Q\to N)$ an $L^d$-Loewner chain. Then $(f_t\circ h_t\colon M\to N)$ is an $L^d$-Loewner chain.
\end{lemma}
\begin{proof}
It is clear that $(f_t\circ h_t\colon M\to N)$ is an algebraic Loewner chain. Let $T>0$ and let  $K\subset M$ be a compact set. The set $\hat K:=\bigcup_{0\leq t\leq T}h_t(K)\subset Q$ is compact by Lemma \ref{conti}, and the family  $(f_t)_{0\leq t\leq T}$ is equi-Lipschitz on $\hat K$ (see (\ref{estimatelip})). Thus if $0\leq s\leq t\leq T$ and $z\in K$,
\begin{align}
 d_N(f_t(h_t(z)),f_s(h_s(z)))&\leq d_N(f_t(h_t(z)),f_t(h_s(z)))+d_N(f_t(h_s(z)),f_s(h_s(z)))\\ \nonumber
&\leq  L(\hat K,T)d_Q(h_t(z),h_s(z))       +\int_s^t k_{\hat K,T}(\xi) d\xi\\\nonumber
&\leq  L(\hat K,T) \int_s^t h_{ K,T}(\xi) d\xi +\int_s^t k_{\hat K,T}(\xi) d\xi. \nonumber
\end{align}
\end{proof}
\begin{remark}
Two conjugated $L^d$-evolution families have essentially the same associated $L^d$-Loewner chains. Namely, if $(g_t\colon Q\to N)$ is an $L^d$-Loewner chain associated with $(\psi_{s,t})$ which is conjugate to  $(\v_{s,t})$ on $M$ via $(h_t\colon M\to Q)$, then $(g_t\circ h_t\colon M\to N)$ is an $L^d$-Loewner chain associated with $(\v_{s,t})$. In particular, $\Lr(\v_{s,t})$ is biholomorphic to $\Lr(\psi_{s,t}).$
\end{remark}
\begin{proposition}\label{L^dconj}
Let $(\psi_{s,t})$ be an $L^d$-evolution family on a complete hyperbolic complex manifold $Q$ and let $(h_t\colon M\to Q)$ be a family of intertwining mappings of order $d$. Then the family $(h_t^{-1}\circ\psi_{s,t}\circ h_s)$ is an  $L^d$-evolution family on $M$.
\end{proposition}
\begin{proof}
By Theorem \ref{ev-to-Low}, there exists an $L^d$-Loewner chain $(f_t\colon Q\to N)$ associated with $(\psi_{s,t})$. By Lemma \ref{intercomp} the family $(f_t\circ h_t)$ defines an $L^d$-Loewner chain $(g_t\colon M\to N)$.  Then
$$h_t^{-1}\circ\psi_{s,t}\circ h_s=h_t^{-1}\circ f_t^{-1}\circ f_s \circ h_s=g_t^{-1}\circ g_s,$$ which by Theorem \ref{Low-to-ev} is an $L^d$-evolution family on $M$.
\end{proof}

Let now $M$ be the unit ball $\B^n$.
\begin{definition}
Take $a\in\B^n$. Let $P_a(z):=\frac{\la z,a\ra}{\|a\|^2}a$ for $a\neq 0$, $P_0=0$,
$Q_a(z):=z-P_a(z)$ and $s_a:=(1-\| a\|^2)^{1/2}$. Then
\[
\varphi_a(z):=\frac{a-P_a(z)-s_aQ_a(z)}{1-\langle z,a\rangle}
\]
is an automorphism of the ball $\B^n$ (see, {\sl e.g.},
\cite{abate} or \cite{R}).
\end{definition}
We can now show that in order to study the Loewner range of an $L^d$-evolution family  on $\B^n$ one can assume that it fixes the origin.
\begin{corollary}
Let $(\psi_{s,t})$ be an $L^d$-evolution family on $\B^n$. There exists a conjugate $L^d$-evolution family $(\v_{s,t})$ such that $$\v_{s,t}(0)=0,\quad 0\leq s\leq t.$$
\end{corollary}
\begin{proof}
Set $a(t):=\psi_{0,t}(0)$. Since $$\|\v_{a(t)}(w)-\v_{a(s)}(w)\|\leq C(K,T)\|a(t)-a(s)\|,\quad w\in K,\ 0\leq s\leq t\leq T,$$ the family $(\v_{a(t)})$ is a  family of intertwining mappings of order $d$.
Define $$\v_{s,t}:=\v_{a(t)}^{-1}\circ \psi_{s,t}\circ \v_{a(s)},$$ which is an $L^d$-evolution family by Proposition \ref{L^dconj}. Since $\v_{a(t)}(0)=a(t),$ we have $\v_{0,t}(0)=0$ for all $t\geq 0$, and by the evolution property $\v_{s,t}(0)=0$ for all $0\leq s\leq t$.
\end{proof}

\section{Extension of Loewner chains from lower dimensional
balls}\label{ss62} The following result provides examples of
$L^d$-Loewner chains on the Euclidean unit ball $\B^n$ in
$\mathbb{C}^n$, which are generated by the Roper-Suffridge
extension operator \cite{RS}. This operator preserves convexity
(see \cite{RS}), starlikeness and the notion of parametric
representation (see e.g. \cite{GK03} and the references
therein).

\begin{theorem}
\label{t1}
Let $d\in [1,+\infty]$ and $(f_t,\D,\C)$ be an $L^d$-Loewner
chain on the unit disc $\D$ such that $|\arg f_t'(0)|<\pi/2$ and
$|\arg(f_s'(0)/f_t'(f_t^{-1}\circ f_s(0)))|<\pi/2$ for $t\geq s\geq 0$.
Also let $(F_t\colon \B^n\to \C^n)$ be given by
\begin{equation}
\label{1}
F_t(z)=\Big(f_t(z_1),\tilde{z}e^{t/2}\sqrt{f_t'(z_1)}\Big),\quad
z=(z_1,\tilde{z})\in \B^n,\quad t\geq 0.
\end{equation}
Then $(F_t)$ is an $L^d$-Loewner chain.
\end{theorem}

\begin{proof}
It is easy to see that $F_t$ is univalent on $\B^n$ for
$t\geq 0$.
Let $(\varphi_{s,t})$ be the
$L^d$-evolution family associated with $(f_t)$ (see Theorem
\ref{Low-to-ev} or \cite{BCD08}). Also let
$\Phi_{s,t}:\B^n\to\mathbb{C}^n$ be given by
$$\Phi_{s,t}(z)=\Big(\varphi_{s,t}(z_1),\tilde{z}e^{(s-t)/2}\sqrt{\varphi_{s,t}'
(z_1)}\Big),\quad
z=(z_1,\tilde{z})\in \B^n,\, t\geq s\geq 0.$$ Then $\Phi_{s,t}$
is a univalent mapping on $\B^n$ and in view of the
Schwarz-Pick lemma, we have
$$\|\Phi_{s,t}(z)\|^2=|\varphi_{s,t}(z_1)|^2+\|\tilde{z}\|^2e^{s-t}|
\varphi_{s,t}'(z_1)|$$
$$<
|\varphi_{s,t}(z_1)|^2+(1-|z_1|^2)e^{s-t}\cdot\frac{1-|\varphi_{s,t}(z_1)|^2}{1-
|z_1|^2}\leq 1,
\, z\in \B^n,\, t\geq s\geq 0.$$ Hence
$\Phi_{s,t}(\B^n)\subseteq \B^n$, and since
$F_s(z)=F_t(\Phi_{s,t}(z))$ for $z\in \B^n$ and $t\geq s\geq
0$, we obtain that $F_s(\B^n)\subseteq F_t(\B^n)$ for $s\leq
t$. In view of the above relations, we deduce that $(F_t)$ is an
algebraic Loewner chain and $(\Phi_{s,t})$ is the associated algebraic evolution
family.

It remains to prove that $(F_t)$ is of order $d$. Since
$(\varphi_{s,t})$ is an evolution family of order $d$, we deduce in view of
\cite[Theorem 6.2]{BCD08} that there
exists a Herglotz vector field $g(z_1,t)$ of order $d$ such
that
$$\frac{\partial \varphi_{s,t}}{\partial t}(z_1)=g(\varphi_{s,t}(z_1),t),\quad
\mbox{ a.e. }\quad t\in [s,+\infty),\quad \forall z_1\in \D.$$
Now, let $G=G(z,t):\B^n\times \R^+\to\mathbb{C}^n$ be
given by
$$G(z,t)=\Big(g(z_1,t),\frac{\tilde{z}}{2}(-1+g'(z_1,t))\Big),
\quad z=(z_1,\tilde{z})\in \B^n,\, t\geq 0.$$ Then $G(z,t)$ is
a weak holomorphic vector field of order $d$ on $\B^n$. Indeed,
the measurability and holomorphicity conditions from the
definition of a weak holomorphic vector field are satisfied. We
next prove that for each $r\in (0,1)$ and $T>0$, there exists
$C_{r,T}\in L^d([0,T],\mathbb{R})$ such that
$$\|G(z,t)\|\leq C_{r,T}(t),\quad \|z\|\leq r,\quad t\in [0,T].$$
But the above condition can be easily deduced by using the fact
that $g(z_1,t)$ is a Herglotz vector field of order $d$ on $\D$
and by the Cauchy integral formula.

On the other hand, since $\varphi_{s,t}$ is locally absolutely
continuous on $[s,+\infty)$ locally uniformly with respect to
$z_1\in \D$, it follows in view of Vitali's theorem (see e.g.
\cite[Chapter 6]{Po}) that
$$\frac{\partial}{\partial t}\Big(\frac{\partial \varphi_{s,t}(z_1)}{\partial
z_1}\Big)=
\frac{\partial }{\partial z_1}\Big(\frac{\partial \varphi_{s,t}(z_1)}{\partial
t}\Big),
\quad \mbox{ a.e. }\quad t\geq s,\quad \forall z_1\in \D.$$
Using the above equality, we obtain by elementary
computations that
$$\frac{\partial \Phi_{s,t}(z)}{\partial t}=G(\Phi_{s,t}(z),t),\quad \mbox{ a.e. }\quad
t\geq s,\quad
\forall z\in \B^n.$$
Therefore, as in the proof of \cite[Proposition 2]{BCD09},
we deduce that
$(dk_M)_{(z,w)}(G(z,t),G(w,t))\leq 0$ for a.e. $t\in\R^+$,
$z\neq w$.
Hence $G(z,t)$ is a Herglotz vector field of order $d$ on
$\B^n$.
Also, as in the proof of \cite[Proposition 1]{BCD09},
we deduce that $(\Phi_{s,t})$ is an evolution family
of order $d$.
Finally, we
conclude that the associated algebraic Loewner chain $(F_t)$ is of order
$d$ on $\B^n$ by Theorem \ref{ev-to-Low}. This completes the proof.
\end{proof}

\begin{corollary}
\label{c1} Let $f:\D\to\mathbb{C}$ be a univalent function such that
$|\arg f'(0)|<\pi/2$. Assume that $(f_t)$ is an $L^d$-Loewner
chain on $\D$ such that $f_0=f$,
$|\arg f_t'(0)|<\pi/2$ for $t\geq 0$,
and $|\arg(f_s'(0)/f_t'(f_t^{-1}\circ f_s(0)))|<\pi/2$
for $t\geq s\geq 0$.
Then $F=\Phi_{n}(f)$ can be imbedded in a
$L^d$-Loewner chain on $\B^n$, where $\Phi_n$ is the
Roper-Suffridge extension operator,
$$\Phi_n(f)(z)=(f(z_1),\tilde{z}\sqrt{f'(z_1)}),\quad z=(z_1,\tilde{z})\in \B^n.$$
\end{corollary}

\begin{proof}
The desired $L^d$-Loewner chain is given by (\ref{1}).
\end{proof}

\section{Spiral-shapedness and Star-shapedness}
\label{ss63}

\begin{definition}
\label{d3.1} Let $\Omega\subset\mathbb{C}^n$ and let $A\in
L(\mathbb{C}^n,\mathbb{C}^n)$ be such that $m(A)>0$, where
$$m(A)=\min\{\Re\langle A(z),z\rangle: \|z\|=1\}.$$
We say that $\Omega$ is {\sl spiral-shaped with respect to $A$}
if $e^{-tA}(w)\in\Omega$ whenever $w\in\Omega$ and $t\in
\R^+$. If $A=\id$ and $\Omega$ is spiral-shaped with
respect to $\id$, we say that $\Omega$ is {\sl star-shaped}.

If $f$ is a univalent mapping on $\B^n$, then $f$ is called
{\sl spiral-shaped with respect to $A$} if the image domain $\Omega={f(\B^n)}$ is
spiral-shaped with respect to $A$. In addition, if $A=\id$ and
$f$ is spiral-shaped with respect to $\id$, we say that $f$ is
{\sl star-shaped} (see \cite{ERS04}).
\end{definition}

\begin{remark}
It is clear that if $f$ is spiral-shaped with respect to $A$,
then $0\in\overline{f(\B^n)}$. Moreover, if $0\in f(\B^n)$,
then the above definition reduces to the usual definition of
spiral-likeness (respectively star-likeness) (see \cite{Gu} and
\cite{Su}).
\end{remark}

We next present some applications of Theorem \ref{LK-PDE} to
the case $M=\B^n$. The first result provides necessary and
sufficient conditions for a locally univalent mapping on
the unit ball ${\B^n}$ in $\mathbb{C}^n$ to be spiral-shaped,
and thus univalent on $\B^n$.

\begin{remark} We remark that the equivalence between the conditions (i)
and (iii) in Theorem \ref{t3.1} below was first obtained by
Elin, Reich and Shoikhet (see the proof of \cite[Proposition
3.5.2]{ERS04}; cf. \cite[Proposition 3.7.2]{ERS04}; \cite{ReSh})
by different arguments (compare \cite{ERS00}). In the case
$f(0)=0$, the analytic characterization of spiral-likeness is
due to Gurganus \cite{Gu} and Suffridge \cite{Su}.
\end{remark}

\begin{theorem}
\label{t3.1} Let $f:{\B^n}\to\mathbb{C}^n$ be a locally
univalent mapping such that $0\in \overline{f({\B^n})}$.
Also let $A\in L(\mathbb{C}^n,\mathbb{C}^n)$ be such that
$m(A)>0$. Then the following conditions are equivalent:

$(i)$ $f$ is spiral-shaped with respect to $A$;

$(ii)$ The family  $(f_t:=e^{tA}f(z))_{t\geq 0}$ is an
$L^\infty$-Loewner chain.

$(iii)$ $f$ is univalent on $\B^n$ and
\begin{equation}
\label{spiral}
\Re\langle (df_z)^{-1}Af(z),z\rangle\geq (1-\|z\|^2)\Re\langle (df_0)^{-1}Af(0),z\rangle,
\quad z\in {\B^n}.
\end{equation}
\end{theorem}

\begin{proof}
The equivalence between the conditions (i) and (ii) is
immediate. Now, we assume that the condition (ii) holds. Then
$f$ is univalent on ${\B^n}$. Let $G(z,t)$ be the Herglotz
vector field of order $\infty$ given by Corollary \ref{corollaryLK-PDE}. A direct computation from \eqref{LK-PDE} implies
\begin{equation}
\label{4.50}
G(z,t)=-(df_z)^{-1}Af(z),\quad t\geq 0,\quad z\in {\B^n}.
\end{equation}
Since by the very definition a Herglotz vector field is a
semicomplete vector field for almost every $t\geq 0$, it
follows that $-(df_z)^{-1}Af(z)$ is semicomplete. Hence, by
\cite[Proposition 3.5.2]{ERS04} (where the sign convention is
different from the one adopted here), we deduce the relation
(\ref{spiral}), as claimed.

Conversely, assume that the condition (iii) holds. Clearly $(f_t)$ is a
family of univalent mappings on $\B^n$ such that the mapping $t\mapsto f_t(z)$
is of class $C^\infty$ on $\R^+$ for all $z\in \B^n$. Also $(f_t)$ satisfies
the differential equation
\begin{equation}
\label{4.49}
\frac{\partial f_t}{\partial t}(z)=-(df_t)_zG(z,t),\quad \mbox{ a.e. }\quad
t\geq 0,\quad \forall\ z\in {\B^n},
\end{equation}
where $G(z,t)$ is given by \eqref{4.50}. In view of the
condition (\ref{spiral}) and \cite[Lemma 3.3.2]{ERS04}, we
deduce that the mapping $G(z,t)$ is a semicomplete vector field
for all $t\geq 0$, and thus it is a Herglotz vector field of
order $\infty$ by \cite[Theorem 0.2]{BCD10}. Hence $(f_t)$ is
an $L^\infty$-Loewner chain by Theorem \ref{LK-PDE}. This
completes the proof.
\end{proof}

We next give the following analytic characterization of
star-shapedness on the unit ball $\B^n$ (cf. \cite{ERS04}). In
the case $f(0)=0$, the inequality in the third statement
becomes the well known analytic characterization of
star-likeness for locally univalent mappings on $\B^n$
(see \cite{Go}, \cite{GK03}, \cite{Su}  and the
references therein).  Necessary and sufficient conditions for star-likeness with respect to a
boundary point are given in \cite{Li-St}.

\begin{corollary}
\label{c-star} Let $f:{\B^n}\to\mathbb{C}^n$ be a locally
univalent mapping such that $0\in \overline{f({\B^n})}$.
Then the following conditions are equivalent:

$(i)$ $f$ is star-shaped;

$(ii)$ The family  $(f_t:=e^tf(z))_{t\geq 0}$ is an
$L^\infty$-Loewner chain.

$(iii)$ $f$ is univalent on ${\B^n}$ and
$$\Re\langle (df_z)^{-1}f(z),z\rangle\geq (1-\|z\|^2)\Re\langle (df_0)^{-1}f(0),z\rangle,
\quad z\in {\B^n}.$$
\end{corollary}

\begin{corollary}
\label{c2} Let $f:\D\to\mathbb{C}$ be a star-shaped function on
$\D$ such that $|\arg f'(0)|<\pi/2$ and
$|\arg(f'(0)/f'(f^{-1}(\lambda f(0)))|<\pi/2$ for $\lambda\in
(0,1]$. Also let $F=\Phi_n(f)$. Then $F$ is also star-shaped on
$\B^n$.
\end{corollary}
\begin{proof}
Since $f$ is star-shaped, it follows that $f_t(z_1)=e^tf(z_1)$
is an $L^\infty$-Loewner chain by Corollary \ref{c-star}.
Let $(F_t)$ be the chain given by (\ref{1}).
In view of Theorem \ref{t1}, $(F_t)$ is an
$L^\infty$-Loewner chain on $\B^n$.
Moreover, since $0\in \overline{F(\B^n)}$
and $F_t(z)=e^tF(z)$, we deduce that the mapping $F=F_0$ is
star-shaped on $\B^n$, by Corollary \ref{c-star}.
This completes the proof.
\end{proof}

\end{document}